\documentclass[12pt]{amsart}
\usepackage{color}
\usepackage{amsmath}
\usepackage[colorlinks=true]{hyperref}

\usepackage{amscd,amsthm,amsfonts,amssymb,esint}
\usepackage{fullpage}
\usepackage[all]{xy}
\usepackage{mathtools}
\usepackage{xcolor}
\usepackage{mathrsfs}
\usepackage{amsmath}
\usepackage[all]{xy}
\usepackage{geometry}
\usepackage{array}
\usepackage{enumerate}
\usepackage{tikz-cd}
\usepackage{comment}
\usepackage{setspace}

\newcommand\mO{\mathcal{O}}

\newcommand{\mb}[1]{\mathbb{#1}}
\newcommand{\mr}[1]{\mathrm{#1}}

\newcommand{\ov}{\overline}
\newcommand\ra{\rightarrow}

\newcommand\wt{\widetilde}

\newcommand\Pic{\mathrm{Pic\,}}

\newcommand{\mib}[1]{\textit{\textbf{#1}}}

\begin{document}
\begin{spacing}{1.25}

\newtheorem{theorem}{Theorem}[section]
\newtheorem{lemma}[theorem]{Lemma}
\newtheorem{proposition}[theorem]{Proposition}
\newtheorem{corollary}[theorem]{Corollary}
\newtheorem{conjecture}[theorem]{Conjecture}
\newtheorem{conv}[theorem]{Convention}
\newtheorem{ques}[theorem]{Question} 
\newtheorem{prb}[theorem]{Problem}

	\theoremstyle{definition}
	\newtheorem{definition}[theorem]{Definition}
	\newtheorem{example}[theorem]{Example}
	
	\theoremstyle{remark}
	\newtheorem{remark}[theorem]{Remark}
	
	\numberwithin{equation}{section}


	\title{Arithmetic Degrees are Cohomological Lyapunov Multipliers}
	\author{Jiarui Song}
	\address{School of Mathematical Sciences \\ Peking University \\ Beijing 100871 \\ China}
	\email{soyo999@pku.edu.cn}
	\author{Junyi Xie}
	\address{Beijing International Center for Mathematical Research \\ Peking University \\ Beijing 100871 \\ China}
	\email{xiejunyi@bicmr.pku.edu.cn}
	\author{She Yang}
	\address{Beijing International Center for Mathematical Research \\ Peking University \\ Beijing 100871 \\ China}
	\email{ys-yx@pku.edu.cn}

	\begin{abstract}
For endomorphisms of projective varieties, we prove that the arithmetic degree of a point with Zariski dense orbit must be a cohomological Lyapunov multiplier of the dynamical system. We will apply our result to deduce a corollary towards the dynamical Mordell--Lang conjecture.
	\end{abstract}
	
	\maketitle
	\setcounter{tocdepth}{1}
	
\section{Introduction}\label{Sec_intro}

Let $f:X\dashrightarrow X$ be a dominant rational self-map of a projective variety $X$ defined over $\overline{\mb{Q}}$. Fix a Weil height $h_H$ associated to an ample line bundle $H$ on $X$. For a point $x\in X(\ov{\mb{Q}})$ whose $f$-orbit $\mO_f(x)=\{x,f(x),f^2(x),\ldots\}$ is well-defined, the \emph{arithmetic degree} of $x$ is defined by 
\[\alpha_f(x)=\lim_{n\ra\infty}\max\{1, h_H(f^n(x))\}^{\frac{1}{n}},\]
provided the limit exists. It is conjectured by Kawaguchi and Silverman that this limit always exists, and furthermore, that $\alpha_f(x)=\lambda_1(f)$, the first dynamical degree of $f$, whenever the $f$-orbit $\mO_f(x)$ is Zariski dense in $X$, see \cite{KS16C}. For recent advances on this conjecture, see \cite{Mat22}. If the $f$-orbit $\mO_f(x)$ is generic, meaning that it intersects every proper closed subset of $X$ in only finitely many points, then it was proved in \cite[Theorem 2.2]{Mat25} that $\alpha_f(x)$ exists and coincides with one of the cohomological Lyapunov multipliers $\mu_i(f)$ of $f$.

In this paper, we study the behavior of arithmetic degrees of endomorphisms under the weaker assumption that the the orbit is Zariski dense in $X$, rather than generic. In the case where $f: X\ra X$ is an endomorphism, it was shown in \cite{KS16} that $\alpha_f(x)$ exists, and it is equal to the modulus of an eigenvalue of the linear map $f^*:\mathrm{N}^1(X)_{\mathbb{R}}\rightarrow\mathrm{N}^1(X)_{\mathbb{R}}$, where $\mathrm{N}^1(X)$ is the numerical group of line bundles on $X$. We strengthen this result by proving that, if $\mO_f(x)$ is Zariski dense in $X$, then $\alpha_f(x)$ must be one of the cohomological Lyapunov multipliers of $f$.

Previous works on arithmetic degrees (see for example \cite{KS16,KS16C,Mat20}) have largely focused on the case where the base field is a number field, as there is a natural notion of Weil heights. In \cite{LS24}, a generalization to arbitrary fields of characteristic zero was introduced via spread-out techniques and the notion of Moriwaki heights \cite{Mor00}. This broader framework allows one to address problems in arithmetic dynamics over fields such as $\mathbb{C}$, including, for example, applications to the dynamical Mordell–Lang conjecture (see Section \ref{Sec_app}). For generality, our results are proved over finitely generated fields, or more generally, over arbitrary fields of characteristic zero.

\subsection*{Settings of this article.}

\ 

We start with a normal projective variety $X_{\mathbb{C}}$, a surjective endomorphism $f_{\mb{C}}:X_{\mathbb{C}}\ra X_{\mathbb{C}}$ over $\mathbb{C}$ and a point $x_{\mb{C}}\in X(\mb{C})$. Let $K$ be a finitely generated field over $\mathbb{Q}$ such that the coefficients of defining equations of $X_{\mathbb{C}}$ and $f_{\mb{C}}$, as well as the coordinates of $x_{\mb{C}}$, are all contained in $K$. In other words, there exists a projective variety $X$, a surjective endomorphism $f: X\ra X$ over $K$, and a point $x\in X(K)$ such that $(X,f,x)$ is a model of $(X_{\mathbb{C}},f_{\mb{C}}, x_{\mb{C}})$. We fix an ample line bundle $L$ on $X$. Then we can measure the complexity of the dynamical system $(X,f)$ and the orbit $\mathcal{O}_f(x)$.
\begin{enumerate}
\item
For $i\in\{0,\dots,\mathrm{dim}(X)\}$, the $i$-th \emph{dynamical degree} $\lambda_i(f)$ of $(X,f)$ is defined as $\lim\limits_{n\rightarrow\infty}((f^n)^{*}L^i\cdot L^{\mathrm{dim}(X)-i})^{\frac{1}{n}}$. The limits exist and are independent of the choice of $L$. See \cite{DS05,Dan20,Tru20}.

\item
Fix a Moriwaki height function (see subsection \ref{subsec_arithdeg}) $h_{L}:X(\overline{K})\rightarrow\mathbb{R}_{\geq1}$. Then the \emph{arithmetic degree} $\alpha_f(x)$ is defined as $\lim\limits_{n\rightarrow\infty}h_{L}(f^n(x))^{\frac{1}{n}}$. The limit exists and is independent of the choice of $L$, $h_L$ and the field $K$. See \cite{KS16,Oh22,LS24}.
\end{enumerate}

\begin{remark}
When $K$ is a number field, the arithmetic degree is defined directly using a Weil height function.
\end{remark}

In \cite{xie24}, the second-named author introduced a notion of \emph{cohomological Lyapunov multipliers} of the dynamical system $(X,f)$. For $i\in\{1,\dots,\mathrm{dim}(X)\}$, the $i$-th cohomological Lyapunov multiplier $\mu_i(f)$ is defined as $\frac{\lambda_i(f)}{\lambda_{i-1}(f)}$. We refer to Section \ref{Sec_pre} for more properties of these concepts.

\ 

Our main result is as follows. We continue with the previous notions.

\begin{theorem}\label{Thm_main}
We have $\alpha_f(x)\in\{\mu_1(f),\dots,\mu_{\mathrm{dim}(X)}(f)\}\cap\mathbb{R}_{\geq1}$, provided that the orbit $\mathcal{O}_f(x)$ is Zariski dense in $X$.
\end{theorem}

From Theorem \ref{Thm_main}, we deduce the following corollary about the Kawaguchi--Silverman conjecture.

\begin{corollary}
Let $X$ be a normal projective variety and $f:X\ra X$ be a surjective endomorphism of $X$ with $\lambda_1(f)>\lambda_2(f)$. Then the Kawaguchi--Silverman conjecture holds. In other words, if $\mO_f(x)$ is Zariski dense in $X$, then $\alpha_f(x)=\lambda_1(f)$. 
\end{corollary}

\begin{remark}
\begin{enumerate}
\item
We have assumed that $X$ is normal and geometrically connected for simplicity. These conditions are mild and are usually easy to fulfill in the applications.

\item
It is proved in \cite{KS16} that $\alpha_f(x)$ is the modulus of an eigenvalue of the linear map $f^*:\mathrm{N}^1(X)_{\mathbb{R}}\rightarrow\mathrm{N}^1(X)_{\mathbb{R}}$, where $\mathrm{N}^1(X)$ is the numerical group of line bundles on $X$. In fact, all of the cohomological Lyapunov multipliers $\mu_i(f)$ are real and positive eigenvalues of this map (see Theorem \ref{Thm_Xie1.4}). So our theorem is a strengthening of Kawaguchi--Silverman's result.

\item According to the Kawaguchi--Silverman conjecture, one expects that $\alpha_f(x)=\mu_1(f)=\lambda_1(f)$ when $\mO_f(x)$ is Zariski dense in $X$.
\end{enumerate}
\end{remark}

As an application, we can prove the following result of dynamical Mordell--Lang type. It is another example of applying height arguments toward the DML conjecture, following the spirit of \cite{XY}. The point in the corollary is that the speeds of height growth of $\mathcal{O}_f(x)$ and $\mathcal{O}_g(y)$ must be different. One could compare this corollary with \cite[Proposition 4.1]{XY}.

\begin{corollary}\label{Cor_mordelllang}
Let $X$ and $Y$ be projective varieties over $\mathbb{C}$. Let $f$ and $g$ be surjective endomorphisms of $X$ and $Y$, respectively. Let $x\in X(\mathbb{C})$ and $y\in Y(\mathbb{C})$ be closed points and let $V\subseteq X\times Y$ be a positive dimensional irreducible closed subvariety. Suppose that
\begin{enumerate}
\item
$\{\mu_1(f),\dots,\mu_{\mathrm{dim}(X)}(f)\}\cap\{\mu_1(g),\dots,\mu_{\mathrm{dim}(Y)}(g)\}\cap\mathbb{R}_{\geq1}=\emptyset$;
\item
the orbits $\mathcal{O}_f(x)$ and $\mathcal{O}_g(y)$ are dense in $X$ and $Y$, respectively; and
\item
both of the projection maps $V\rightarrow X$ and $V\rightarrow Y$ are generically finite onto their image.
\end{enumerate}
Then $V\cap\mathcal{O}_{f\times g}((x,y))$ cannot be dense in $V$.
\end{corollary}

The structure of this article is as follows. In Section \ref{Sec_pre}, we review some properties of dynamical degrees, cohomological Lyapunov multipliers, and arithmetic degrees. Then we prove Theorem \ref{Thm_main} in Section \ref{Sec_proof}. Finally, we prove Corollary \ref{Cor_mordelllang} in Section \ref{Sec_app}.

\section{Preparations}\label{Sec_pre}
In subsection \ref{subsec_dyndeg}, we recall some properties of dynamical degrees and cohomological Lyapunov multipliers. In subsection \ref{subsec_arithdeg}, we recall the notion of Moriwaki height and some properties of the arithmetic degrees.

\subsection{Dynamical degrees and cohomological Lyapunov multipliers}\label{subsec_dyndeg}

\ 

In this subsection, we work over an arbitrary base field $K$. For safety, all the varieties $X$ (and $Y$, etc.) are assumed to be geometrically integral and projective over $K$. We fix a surjective endomorphism $f$ of $X$. We have defined the dynamical degrees $\lambda_i(f)$ and the cohomological Lyapunov multipliers $\mu_i(f)$ in Section \ref{Sec_intro}. Now we recall some properties.

We start with the following proposition. Part (i) follows from the fact that dynamical degrees are log-concave, while part (ii) is a consequence of the relative dynamical degree formula \cite[Theorem 1.1]{DN11} (see also \cite[Theorem 4]{Dang} and \cite[Theorem 1.3]{Tru20}).

\begin{proposition}\label{Prop_Lya}
Let $g$ be a surjective endomorphism of $Y$ and let $\pi:X\rightarrow Y$ be a surjective morphism such that $\pi\circ f=g\circ\pi$. Then the following holds.
\begin{enumerate}
\item
$\mu_1(f)\geq\cdots\geq\mu_{\mathrm{dim}(X)}(f)>0$.
\item
$\{\mu_1(g),\dots,\mu_{\mathrm{dim}(Y)}(g)\}\subseteq\{\mu_1(f),\dots,\mu_{\mathrm{dim}(X)}(f)\}$ as multiple-sets.
\end{enumerate}
\end{proposition}

In particular, in the above setting, we see that $\lambda_1(g)=\mu_1(g)$ is also a cohomological Lyapunov multiplier of $f$.

Next, we recall the key properties of cohomological Lyapunov multipliers, following \cite{xie25}. We denote $\mathrm{N}^1(X)$ as the numerical class group of line bundles on $X$. The theorem of the base guarantees that $\mathrm{N}^1(X)$ is a finite free $\mathbb{Z}$-module. We denote $f^*:\text{N}^1(X)_{\mathbb{R}}\rightarrow\text{N}^1(X)_{\mathbb{R}}$ as the pull-back map induced by $f$.

The following theorem is \cite[Theorem 1.4]{xie25}.

\begin{theorem}\label{Thm_Xie1.4}
We have $\{\mu_1(f),\dots,\mu_{\mathrm{dim}(X)}(f)\}=\{\alpha\in\mathbb{R}|\ \mathrm{Im}(f^*-\alpha)\cap\mathrm{Big}(X)=\emptyset\}$, in which $\mathrm{Big}(X)\subseteq\mathrm{N}^1(X)_{\mathbb{R}}$ is the big cone of $X$. In particular, all of the $\mu_i(f)$ are eigenvalues of $f^*$.
\end{theorem}

The theorem below is a combination of Theorem \ref{Thm_Xie1.4} and \cite[Theorem 1.3]{xie25}.

\begin{theorem}\label{Thm_Xie1.3}
The linear subspace $\sum\limits_{i=1}^{\mathrm{dim}(X)}\mathrm{ker}(f^*-\mu_i(f))^{\rho}\subseteq\mathrm{N}^1(X)_{\mathbb{R}}$ has a nonempty intersection with $\mathrm{Big}(X)$, in which $\rho$ is the rank of $\mathrm{N}^1(X)$.
\end{theorem}

Now we shall analyze the first dynamical degree of endomorphisms of abelian varieties for future uses.

\begin{definition}
Let $A$ be an abelian variety and let $f$ be an algebraic group endomorphism of $A$. We define the spectral radius of $f$, denoted by $\rho(f)$, as the maximum of the absolute values of roots of the minimal polynomial of $f$.
\end{definition}

\begin{proposition}\label{Prop_abel}
Let $A$ be an abelian variety and let $f$ be a self-isogeny of $A$.
\begin{enumerate}
\item
Let $f^{\vee}:A^{\vee}\rightarrow A^{\vee}$ be the dual isogeny of $f$. Then $\rho(f)=\rho(f^{\vee})$.
\item
Let $\tau_a$ be a translation of $A$ by some point $a\in A(K)$. Then $\lambda_1(f)=\lambda_1(\tau_a\circ f)$.
\item
Suppose that the base field $K$ has characteristic 0. Then $\rho(f)\leq\sqrt{\lambda_1(f)}$.
\end{enumerate}
\end{proposition}

\begin{proof}
We may assume that the base field $K$ is algebraically closed. Then part (1) is true because $f$ and $f^{\vee}$ have the same minimal polynomial. Part (2) is true because pull-back by translations do not change the numerical class of line bundles. To prove part (3), we may assume that $K=\mathbb{C}$ by the Lefschetz principle.

Now every root of the minimal polynomial of $f$ is an eigenvalue of the linear map $f^*:H^{1}(A,\mathbb{C})\rightarrow H^{1}(A,\mathbb{C})$. We know $H^{1}(A,\mathbb{C})=H^{1,0}(A)\oplus H^{0,1}(A)$. So for every root $\alpha$ of the minimal polynomial of $f$, we see that $\alpha\bar{\alpha}=|\alpha|^2$ is an eigenvalue of the linear map $f^*:H^{1,1}(A)\rightarrow H^{1,1}(A)$. But in fact $\lambda_1(f)$ is the spectral radius of $f^*:H^{1,1}(A)\rightarrow H^{1,1}(A)$ (for example, see \cite[top of p. 72]{Tru25}). Hence the result follows.
\end{proof}

\subsection{Arithmetic degrees}\label{subsec_arithdeg}

\ 

In Introduction, we have described the concept of arithmetic degrees. In this subsection, we will recall some of its properties. But firstly, we shall review the concept of Moriwaki heights. For short, it is a height machinery over finitely generated fields that admits all good properties as for the Weil height over number fields. We will only make a brief review and refer to \cite{Mor00} for details.

Let $K$ be a finitely generated field over $\mathbb{Q}$. Let $B$ be a normal arithmetic variety, which is a model of $K$. We fix a big and nef hermitian line bundle $\overline{H}$ on $B$. Let $X$ be a projective variety over $K$. Then we can define a group homomorphism from $\mathrm{Pic}(X)$ to the $\mathbb{R}$-vector space $\frac{\{\textrm{functions}\ X(\overline{K})\rightarrow\mathbb{R}\}}{\{\textrm{bounded functions}\}}$ (denote as $L\mapsto h_L$) in the following way.

Let $L\in\mathrm{Pic}(X)$. Let $\pi:\mathcal{X}\rightarrow B$ be a model of $X$, in which $\mathcal{X}$ is an arithmetic variety. By choosing $\mathcal{X}$ carefully, we can find a line bundle $\mathcal{L}\in\mathrm{Pic}(\mathcal{X})$ that restricts to $L\in\mathrm{Pic}(X)$. We form $\mathcal{L}$ into a hermitian line bundle $\overline{\mathcal{L}}$ over $\mathcal{X}$. Then we define $h_L$ by the formula $h_L(x)=\frac{\overline{\mathcal{L}}\cdot\pi^*\overline{H}^d\cdot\overline{x}}{[K(x):K]}$ for every $x\in X(\overline{K})$, in which $\overline{x}$ is the closure of $x$ in $\mathcal{X}$ and $d$ is the transcendence degree of $K/\mathbb{Q}$. This notion is well-defined (i.e. independent of the choice of $(\mathcal{X},\overline{\mathcal{L}})$) up to bounded functions. Notice that we have fixed $(B,\overline{H})$ in the procedure above. Since we will never modify these data, we do not emphasize them in the notion.

By tensoring $\mathbb{R}$, we get an $\mathbb{R}$-linear map $\mathrm{Pic}(X)_{\mathbb{R}}\rightarrow\frac{\{\textrm{functions}\ X(\overline{K})\rightarrow\mathbb{R}\}}{\{\textrm{bounded functions}\}}$, which we still denote as $L\mapsto h_L$. We list some properties of this map, following \cite[Proposition 3.3.7]{Mor00}. Notice that the projection formula here is a direct consequence of the projection formula in arithmetic intersection theory.

\begin{proposition}\label{Prop_ht}
Let $L\in\mathrm{Pic}(X)$.
\begin{enumerate}
\item
The function $h_L$ is bounded below on $(X\backslash\mathrm{Bs}(L))(\overline{K})$ in which $\mathrm{Bs}(L)$ is the base locus of $L$.

\item (Northcott)
If $L$ is ample, then the set $\{x\in X(\overline{K})|\ [K(x):K]\leq d,h_L(x)\leq C\}$ is finite for every positive integer $d$ and every $C>0$ (and every representative $h_L$). 

\item (projection formula)
Let $Y$ be another projective variety over $K$ and let $f:X\rightarrow Y$ be a morphism. Then for every $L\in\mathrm{Pic}(Y)_{\mathbb{R}}$, we have $h_{f^*L}(x)=h_L(f(x))+O(1)$ as functions on $X(\overline{K})$.
\end{enumerate}
\end{proposition}

Now we return to the setting in the Introduction and review some properties of the arithmetic degree. Let $X$ be a geometrically connected normal projective variety over $K$. Let $f$ be a surjective endomorphism of $X$, and let $x\in X(\overline{K})$ be a point such that the orbit $\mathcal{O}_f(x)$ is Zariski dense in $X$. One can learn the following results from \cite{KS16,Sil17,Oh22}.

\begin{proposition}\label{Prop_arithdeg}
\begin{enumerate}
\item
We have $\alpha_f(x)\leq\lambda_1(f)$.
\item
If $X$ is an abelian variety, then $\alpha_f(x)=\lambda_1(f)$.
\end{enumerate}
\end{proposition}

We state the following lemma for future use. The philosophy is that big divisors can control the ample ones on an open dense set.

\begin{lemma}\label{Lem_bight}
Let $B \in\mathrm{Pic}(X)_{\mathbb{R}}$ be a big $\mathbb{R}$-divisor. Then there exists an infinite sequence of integers $n_1 < n_2 < \cdots$ such that
\[
\alpha_f(x) = \lim_{i \to \infty} h_B^+\big(f^{n_i}(x)\big)^{\frac{1}{n_i}},
\]
where $h_B^+ := \max\{h_B, 1\}$. Furthermore, the set $\{h_B^+(f^{n_i}(x)) \mid i \in \mathbb{Z}_{>0} \}$ is not bounded above.
\end{lemma}

\begin{proof}
Since $B$ is big, there exist divisors $A, E \in \mathrm{Pic}(X)_{\mathbb{R}}$ such that $A$ is ample, $E$ is effective, and $B = A + E$.

Since the $f$-orbit of $x$ is Zariski dense in $X$, there exists a sequence $n_1 < n_2 < \cdots$ such that $f^{n_i}(x) \notin \operatorname{Supp}(E)(K)$ for all $i$. For such $n_i$, we have
\[
h_B(f^{n_i}(x)) = h_A(f^{n_i}(x)) + h_E(f^{n_i}(x)) + O(1) \geq h_A(f^{n_i}(x)) + O(1).
\]
It follows that
\[
\lim_{i \to \infty} h_B^+\big(f^{n_i}(x)\big)^{\frac{1}{n_i}} \geq \lim_{i \to \infty} h_A^+\big(f^{n_i}(x)\big)^{\frac{1}{n_i}} = \alpha_f(x).
\]

On the other hand, since $A$ is ample, there exists $m > 0$ such that $mA - B$ is ample. Then
\[
h_{mA}(f^{n_i}(x)) \geq h_B(f^{n_i}(x)) + O(1),
\]
and thus,
\[
\lim_{i \to \infty} h_B^+\big(f^{n_i}(x)\big)^{\frac{1}{n_i}} \leq \lim_{i \to \infty} h_{mA}^+\big(f^{n_i}(x)\big)^{\frac{1}{n_i}} = \alpha_f(x).
\]

Combining the inequalities above, we conclude that
\[
\lim_{i \to \infty} h_B^+\big(f^{n_i}(x)\big)^{\frac{1}{n_i}} = \alpha_f(x).
\]

Finally, suppose for contradiction that $h_B^+(f^{n_i}(x)) \leq C_1$ for all $i$ and some constant $C_1 > 0$. Then $h_A(f^{n_i}(x)) \leq C_2$ for some $C_2$ by the comparison above. By the Northcott property (applied to the ample divisor $A$), this would imply that the set $\{f^{n_i}(x)\}$ is finite. So $x\in \mr{Prep}(f)$, contradicting the Zariski density of the $f$-orbit of $x$. Hence, $\{h_B^+(f^{n_i}(x))\}$ is not bounded above.
\end{proof}

\section{Proof of Theorem \ref{Thm_main}}\label{Sec_proof}

In this section, we prove Theorem \ref{Thm_main}. We follow the setting in the Section \ref{Sec_intro}. Let $K$ be a finitely generated field over $\mathbb{Q}$ and let $X$ be a geometrically connected normal projective variety over $K$. Let $f$ be a surjective endomorphism of $X$ and let $x\in X(\overline{K})$ be a point with a Zariski dense $f$-orbit. We fix a datum $(B,\overline{H})$ (see subsection \ref{subsec_arithdeg}) to run the Moriwaki height machinery.

Let $\mathrm{Pic}^0(X)$ be the identity component of the Picard scheme of $X$, which is an abelian variety over $K$. Sometimes we do not distinguish $\mathrm{Pic}^0(X)$ with $\mathrm{Pic}^0(X)(K)$. We have an exact sequence $0\rightarrow\mathrm{Pic}^0(X)_{\mathbb{R}}\rightarrow\mathrm{Pic}(X)_{\mathbb{R}}\rightarrow\mathrm{N}^1(X)_{\mathbb{R}}\rightarrow 0$. The endomorphism $f$ induces a pull-back map $f^*$ on each of them. In particular, we have a self-isogeny $f^*:\mathrm{Pic}^0(X)\rightarrow\mathrm{Pic}^0(X)$.

Next, we fix some notations on linear algebra. Let $V$ be a real vector space of dimension $n$, and let $g: V\ra V$ be a linear map. For $\mu\in \mb{R}$, define 
\[E_\mu:=\ker\big((g-\mu\cdot\mr{id})^n:V\ra V\big)\subset V.\]
We will mainly use this notation towards $f^*:\mathrm{N}^1(X)_{\mathbb{R}}\ra\mathrm{N}^1(X)_{\mathbb{R}}$.

We start with a lifting lemma.

\begin{lemma}\label{Lem_lift}
Let $p: \mr{Pic}(X)_{\mb{R}}\ra \mr{N}^1(X)_{\mb{R}}$ be the natural projection. Let $\rho_0$ denote the spectral radius of the self-isogeny $f^*: \mr{Pic}^0(X)\ra \mr{Pic}^0(X)$. Then there exists a lifting $\iota: \bigoplus\limits_{\mu>\rho_0}E_{\mu}\ra \mr{Pic}(X)_{\mb{R}}$ such that $p\circ\iota$ is the natural inclusion and $\iota\circ f^*=f^*\circ\iota$. The last two $f^*$ are pull-backs on $\bigoplus\limits_{\mu>\rho_0}E_{\mu}\subseteq\mathrm{N}^1(X)_{\mathbb{R}}$ and $\mathrm{Pic}(X)_{\mathbb{R}}$, respectively.
\end{lemma}

\begin{proof}
It suffices to construct $\iota$ on each cyclic subspace of $E_\mu$ for $\mu > \rho_0$. Let $E \subset E_\mu$ be a cyclic subspace for $f^*: E_\mu \to E_\mu$, with basis $D_1, D_2, \ldots, D_r$ satisfying
\[f^*D_1\equiv \mu D_1+D_2, f^*D_2\equiv\mu D_2+D_3,\ldots, f^*D_r\equiv\mu D_r.\]

Define the matrix
\[
\Lambda = 
\begin{bmatrix}
\mu & 0 & 0 & \cdots & 0 \\
1 & \mu & 0 & \cdots & 0 \\
\vdots & \ddots & \ddots & \ddots & \vdots \\
0 & \cdots & 1 & \mu & 0 \\
0 & \cdots & 0 & 1 & \mu
\end{bmatrix}_{r\times r},
\]
so that
\[(f^*D_1,f^*D_2,\ldots,f^*D_r)=(D_1,D_2,\ldots,D_r)\Lambda.\]

Choose $L_i\in \mr{Pic}(X)_{\mb{R}}$ such that $p(L_i)=D_i$ for $i=1,2,\ldots,r$, and set $\mib{L}:=(L_1,L_2,\ldots,L_r)\in \mr{Pic}(X)_{\mb{R}}^{r}$. Let $V\subseteq\mr{Pic}(X)_{\mb{R}}$ be the finite dimensional $f^*$-invariant subspace generated by $L_1,\dots,L_r$. Then $\mib{M}:=f^*\mib{L}-\mib{L}\Lambda=(M_1,M_2,\ldots,M_r) \in (V\cap\mr{Pic}^0(X)_{\mb{R}})^r$.

We now show that the sequence $\big((f^*)^n\mib{L}\big)\Lambda^{-n}$ converges in $V^{r}$. Note that
\[
\begin{aligned}
&\big((f^*)^n\mib{L}\big)\Lambda^{-n}-\big((f^*)^{n-1}\mib{L}\big)\Lambda^{-(n-1)}\\
=&\big((f^{*})^{n-1}(f^*\mib{L}-\mib{L}\Lambda)\big)\Lambda^{-n}  \\
=&\big((f^{*})^{n-1}\mib{M}\big)\Lambda^{-n}. \\
\end{aligned}
\]

Notice that $\{(f^*)^n\mib{M}\}_{n\geq0}$ is a linear recurrence sequence, which admits the minimal polynomial of $f^*:\mr{Pic}^0(X)\rightarrow\mr{Pic}^0(X)$ as its characteristic polynomial. So the norm of $\big((f^{*})^{n-1}\mib{M}\big)\Lambda^{-n}$ is $O(n^{r+2g}(\frac{\rho_0}{\mu})^n)$, in which $g=\dim\Pic^0(X)$. It follows that the sequence $\big((f^*)^n\mib{L}\big)\Lambda^{-n}$ converges to some $\wt{\mib{L}}=(\wt{L}_1,\wt{L}_2,\ldots,\wt{L}_r)\in V^{r}$. By construction, we have $p(\wt{\mib{L}})=\mib{D}=(D_1,D_2,\ldots,D_r)\in\mr{N}^1(X)_{\mb{R}}^{r}$ and $f^*\wt{\mib{L}}=\wt{\mib{L}}\Lambda$.

Define $\iota(D_i):=\wt{L}_i, i=1,2,\ldots,r$. This extends linearly to a map $\iota: E\ra \mr{Pic}(X)_{\mb{R}}$. Repeating this construction for suitable cyclic subspaces of $\bigoplus\limits_{\mu>\rho_0}E_{\mu}$, we obtain the desired lifting $\iota:\bigoplus\limits_{\mu>\rho_0}E_{\mu} \ra \mr{Pic}(X)_{\mb{R}}$.
\end{proof}

Next, we give an upper bound of height growth. We first state an elementary lemma and omit its proof, as the proof is verbatim with \cite[Lemma 3.2]{XY}.

\begin{lemma}\label{Lem_poly}
Let $\{x_n\}_{n\geq0}$ be a sequence of complex numbers. Let $P(x)=x^d+a_{d-1}x^{d-1}+\cdots+a_0$ be a polynomial with complex coefficients. Let $\rho=\max\{|z|:P(z)=0\}$. Suppose $|x_{n+d}+a_{d-1}x_{n+d-1}+\cdots+a_0x_n|=O(1)$. Then $|x_n|=O(n^d\rho^n)$ if $\rho\geq1$ and $|x_n|=O(1)$ if $\rho<1$.
\end{lemma}

The next lemma is an immediate consequence of Lemma \ref{Lem_poly}.

\begin{corollary}\label{Cor_htbd}
Let $E\subset \mr{Pic}(X)_{\mb{R}}$ be a subspace invariant under $f^*$. Suppose that there exists a polynomial $P(x)\in\mathbb{R}[x]$ such that $P(f^*)=0$ on $E$. Let $d=\deg(P)$ and let $\rho=\max\{|z|:P(z)=0\}$. Then for any divisor $M\in E$ and any $x\in X(\ov{K})$, we have $|h_M(f^n(x))|=O(n^d\rho^n)$ if $\rho\geq1$ and $|h_M(f^n(x))|=O(1)$ if $\rho<1$.    
\end{corollary}
Now we can prove Theorem \ref{Thm_main}.

\proof[Proof of Theorem \ref{Thm_main}]
Let $d=\dim(X)$. Recall that the cohomological dynamical degrees satisfy
\[\mu_1(f)\geq \mu_2(f)\geq \cdots\geq \mu_d(f)>0.\]
For convenience, we set $\mu_{d+1}(f):=0$.

Let $A = \mathrm{Alb}(X)$ be the Albanese variety of $X$, and let $\pi: X \to A$ be the Albanese morphism. Then $f:X \to X$ induces a morphism $g: A \to A$. Since the orbit $\mO_f(x)$ is Zariski dense, it follows from \cite[Proposition 3.7]{LM21} that $\pi$ is surjective. Consequently, the orbit $\mO_g(\pi(x))$ is Zariski dense in $A$. By Proposition \ref{Prop_arithdeg}(2), we have $\alpha_g(\pi(x))=\lambda_1(g)$, where $\lambda_1(g)$ is the first dynamical degree of $g$. If $A$ is trivial we set $\lambda_1(g):=0$. Since $\pi: X\ra A$ is surjective, we obtain 
\[\alpha_f(x)\geq \alpha_g(\pi(x))=\lambda_1(g).\] 

If $\alpha_f(x)=\lambda_1(g)$, then $\lambda_1(g)\geq 1$, and by Proposition \ref{Prop_Lya}(2), we have $\lambda_1(g)=\mu_1(g)\in \{\mu_1(f), \mu_2(f), \ldots,\mu_d(f)\}$. The conclusion follows. So assume $\alpha_f(x)>\lambda_1(g)$. 

Taking Proposition \ref{Prop_arithdeg}(1) into account, we suppose $\mu_{\ell}(f)>\alpha_f(x)>\mu_{\ell+1}(f)$ for some $\ell\in\{1,\dots,d\}$ by contradiction. Define $E:=\bigoplus\limits_{\mu_i(f)>\sqrt{\lambda_1(g)}}E_{\mu_i(f)}$ and $r:=\dim E$. Choose a basis $\mib{D}=(D_1,D_2,\ldots,D_r)$ of $E$ such that $f^*\mib{D}\equiv \mib{D} \Lambda$, where $\Lambda$ is a matrix of the form
\[
\Lambda = 
\begin{bmatrix}
J_{r_{1,1}}(\mu_1(f)) &   \cdots     & 0               & \cdots          & 0 \\
\vdots         &\ddots       & \vdots              & \cdots          & \vdots\\
0              & \cdots        & J_{r_{2,1}}(\mu_2(f)) & \cdots          & 0 \\
\vdots         &\vdots         & \vdots            & \ddots          & \vdots \\
0              & \cdots       & 0                 & \cdots          & J_{r_{\ell,m}}(\mu_{\ell}(f))
\end{bmatrix},
\]
with each $J_s(\lambda)$ a Jordan block:
\[J_{s}(\lambda) = 
\begin{bmatrix}
\lambda & 0 & 0 & \cdots & 0 \\
1 & \lambda & 0 & \cdots & 0 \\
\vdots & \ddots & \ddots & \ddots & \vdots \\
0 & \cdots & 1 & \lambda & 0 \\
0 & \cdots & 0 & 1 & \lambda
\end{bmatrix}_{s\times s}.\]

Since $\mr{Pic}^0(X)\cong A^{\vee}$, Proposition \ref{Prop_abel} implies $\rho(f^*:\Pic^0(X)\rightarrow\Pic^0(X))\leq \sqrt{\lambda_1(g)}$. Then Lemma \ref{Lem_lift} ensures that there exists a lifting $\iota: E\ra \mr{Pic}(X)_{\mb{R}}$ such that $p\circ \iota$ is the natural inclusion and $\iota\circ f^*=f^*\circ \iota$. Let $\mib{L}=(L_1,L_2,\ldots,L_r):=(\iota(D_1),\iota(D_2),\ldots,\iota(D_r))\in \mr{Pic}(X)_{\mb{R}}^{r}$, so that
\[f^*\mib{L}=\mib{L}\Lambda.\]

We take a vector of Moriwaki height functions $\mib{h}_{\mib{L}}=(h_{L_1},h_{L_2},\ldots, h_{L_r}): X(\ov{K})\ra\mb{R}^{r}$. By the functoriality of Moriwaki heights, we have $\mib{h}_{\mib{L}}\circ f=\mib{h}_{\mib{L}}\Lambda+O(1)$. By \cite[Theorem 5]{KS16}, the sequence $\big(\mib{h}_{\mib{L}}\circ f^n\big)\Lambda^{-n}$ pointwise converges to a canonical height vector $\widehat{\mib{h}}_{\mib{L}}=(\widehat{h}_{L_1},\ldots,\widehat{h}_{L_r}):X(\ov{K})\ra \mb{R}^{r}$ satisfying
\[\widehat{\mib{h}}_{\mib{L}}\circ f=\widehat{\mib{h}}_{\mib{L}}\Lambda,\]
and $\widehat{\mib{h}}_{\mib{L}}=\mib{h}_{\mib{L}}+O(1)$.

By Theorem \ref{Thm_Xie1.3}, we have $\mr{Big}(X)\cap \bigoplus\limits_{1\leq i\leq d}E_{\mu_i(f)}\neq \emptyset$. In other words, there exists a big class $B\in\mr{N}^1(X)_{\mb{R}}$ such that 
\[B=M_1+M_2,\]
where $M_1\in \bigoplus\limits_{1\leq i\leq \ell}E_{\mu_i(f)}$ and $M_2\in \bigoplus\limits_{\ell +1\leq i\leq d}E_{\mu_i(f)}$. Choose any lifting $\wt{M}_2\in \mr{Pic}(X)_{\mb{R}}$ of $M_2$, let $\wt{M}_1=\iota(M_1)$, and define $\wt{B}=\wt{M}_1+\wt{M}_2$. Then $\wt{B}$ is a big divisor.

Suppose $\wt{M}_1=a_1L_1+a_2L_2+\dots+a_rL_r$ with $\mib{a} = (a_1, \ldots, a_r)^{\top} \in \mb{R}^r$. We choose Moriwaki height functions $h_{\wt{B}}, h_{\wt{M}_1}$ and $h_{\wt{M}_2}$ so that $h_{\wt{B}}=h_{\wt{M}_1}+h_{\wt{M}_2}+O(1)$. Define the canonical height \[\widehat{h}_{\wt{M}_1}:=a_1\widehat{h}_{L_1}+a_2\widehat{h}_{L_2}+\dots+a_r\widehat{h}_{L_r},\]
which satisfies $\widehat{h}_{\wt{M}_1}=h_{\wt{M}_1}+O(1)$ and  $\widehat{h}_{\wt{M}_1}\circ f^n=\widehat{\mib{h}}_{\mib{L}}\Lambda^n\mib{a}$. Hence, we can express
\[\widehat{h}_{\wt{M}_1}(f^n(x))=\sum_{\substack{1\leq i\leq \ell\\ 0\leq k\leq r}}c_{i,k}n^k\mu_i^{n}\]
for some $c_{i,k}\in \mb{R}$. 

Next, we estimate the growth of $h_{\wt{M}_2}(f^n(x))$. Let
\[E^{\prime}:=\mr{span}\{(f^*)^n \wt{M}_2\mid n\geq 0\}\subset \mr{Pic}(X)_{\mb{R}},\]
which is a finite-dimensional subspace by \cite[Lemma 19]{KS16}. Since $\wt{M}_2$ lifts $M_2\in \bigoplus\limits_{\ell +1\leq i\leq d}E_{\mu_i}$,  we have
\[\rho_1:=\rho(f^*:E^{\prime}\ra E^{\prime})\leq \max\{\mu_{\ell+1}(f),\rho(f^*)\}\leq \max\{\mu_{\ell+1}(f), \sqrt{\lambda_1(g)}\}<\alpha_f(x),\]
in which $\rho(f^*)$ is the spectral radius of $f^*:\Pic^0(X)\rightarrow\Pic^0(X)$. Then by Corollary \ref{Cor_htbd}, we have $|h_{\wt{M}_2}(f^n(x))|=O(n^{\dim(E')}\rho_1^n)$ if $\rho_1\geq1$ and $|h_{\wt{M}_2}(f^n(x))|=O(1)$ if $\rho_1<1$.

By Lemma \ref{Lem_bight}, there exists a sequence $n_1<n_2<\cdots$ such that $\lim\limits_{j\ra\infty}h_{\wt{B}}^+(f^{n_j}(x))^{\frac{1}{n_j}}=\alpha_f(x)$ and the set $\{h_{\wt{B}}(f^{n_j}(x))\mid j\in \mb{Z}_{>0}\}$ is not bounded above.

If $c_{i,k}=0$ for all $1\leq i\leq \ell$ and $0\leq k\leq r$, then

\[h_{\wt{B}}(f^{n_j}(x))=h_{\wt{M}_2}(f^{n_j}(x))+O(1)=O(\max\{n_j^{\dim(E')}\rho_1^{n_j},1\}),\]

implying
\[\alpha_f(x)=\lim\limits_{j\ra\infty}h_{\wt{B}}^+(f^{n_j}(x))^{\frac{1}{n_j}}\leq \max\{\rho_1,1\}.\]

However, since $\rho_1<\alpha_f(x)$, it follows that $\alpha_f(x)=1$ hence $\rho_1<1$. This implies that $h_{\wt{B}}(f^{n_j}(x))$ is bounded above, hence a contradiction.

If some $c_{i,k}\neq 0$, then for some $1\leq i_0\leq \ell$ and $C_1>0$, we have

\[\widehat{h}_{\wt{M}_1}(f^{n_j}(x))\geq C_1\mu_{i_0}^{n_j},\quad \forall j\gg0.\]

 Since $\rho_1<\alpha_f(x)<\mu_{i_0}$, we find $C_2>0$ such that 

 \[h_{\wt{B}}(f^{n_j}(x))\geq C_2\mu_{i_0}^{n_j}+O(\max\{n_j^{\dim(E')}\rho_1^{n_j},1\})\geq C_2\mu_{i_0}^{n_j},\quad \forall j\gg0.\]

 It follows that

 \[\alpha_f(x)=\lim\limits_{j\ra\infty}h_{\wt{B}}^+(f^{n_j}(x))^{\frac{1}{n_j}}\geq\mu_{i_0}>\alpha_f(x),\]
 which is a contradiction. Thus, we finish the proof.
\endproof

\section{An application towards the dynamical Mordell--Lang conjecture}\label{Sec_app}

We aim to prove Corollary \ref{Cor_mordelllang} in this section. We first briefly introduce the dynamical Mordell--Lang (DML) conjecture. Unless otherwise specified, we let the base field be $\mb{C}$ in this section.

The DML conjecture is one of the core problems in the field of arithmetic dynamics. It asserts that in an algebraic dynamical system, the return set of an orbit into a Zariski closed subset must be a finite union of arithmetic progressions. We refer to \cite{BGT16,xieDML} for more on this problem.

In \cite{XY}, the second-named and third-named authors apply height arguments to study the return set of an orbit into a Zariski closed subset. Corollary \ref{Cor_mordelllang} is another application of this philosophy.

\proof[Proof of Corollary \ref{Cor_mordelllang}]
By a standard argument of considering normalizations, we may assume that $X$ and $Y$ are normal. Let $K\subseteq\mb{C}$ be a finitely generated subfield such that all the data (i.e. $X,Y,f,g,x,y,V$) are defined over $K$. We regard all these data as objects over $K$, and do not change their names by abusing notation. Then the three hypotheses are still valid, and our goal also does not change.

Assume by contradiction that $V\cap\mathcal{O}_{f\times g}((x,y))$ is dense in $V$. Same as in the previous section, we fix a datum $(B,\overline{H})$ for $K$ to run the Moriwaki height machinery. Then we can use Theorem \ref{Thm_main} to see that $\alpha_f(x)\neq\alpha_g(y)$. Without loss of generality, we assume that $1\leq\alpha_f(x)<\alpha_g(y)$. 

Let $p_1:V\ra X$ and $p_2:V\ra Y$ be two projections. We fix ample line bundles $L_1\in\Pic(X)$ and $L_2\in\Pic(Y)$, and fix Moriwaki height functions $h_{L_1}:X(\overline{K})\ra\mb{R}_{\geq1}$ and $h_{L_2}:Y(\overline{K})\ra\mb{R}_{\geq1}$. Since $p_1^*L_1$ is always big, we may choose $L_1$ appropriately such that $p_1^*L_1=A+E$ for some ample $A\in\Pic(V)$ and effective $E\in\Pic(V)$. By definition, we have $\alpha_f(x)=\lim\limits_{n\ra\infty}h_{L_1}(f^n(x))^{\frac{1}{n}}$ and $\alpha_g(y)=\lim\limits_{n\ra\infty}h_{L_2}(g^n(y))^{\frac{1}{n}}$.

Let $S=\{n\in\mathbb{N}|\ (f^n(x),g^n(y))\in(V\backslash\mathrm{Supp}(E))(K)\}$. By our assumption, this is an infinite set. Pick a positive integer $N$ such that $NA-p_2^*L_2$ is ample and fix a Moriwaki height function $h_A:V(\overline{K})\ra\mb{R}_{\geq1}$. Denote $a_n=(f^n(x),g^n(y))$ for short. Then by Proposition \ref{Prop_ht}, we see that the sequences $\{h_{L_1}(f^n(x))-h_A(a_n)\}_{n\in S}$ and $\{Nh_A(a_n)-h_{L_2}(g^n(y))\}_{n\in S}$ are bounded below. Hence $\{Nh_{L_1}(f^n(x))-h_{L_2}(g^n(y))\}_{n\in S}$ is bounded below. But since $S$ is infinite, this contradicts the fact that $1\leq\alpha_f(x)<\alpha_g(y)$. Thus we get a contradiction and finish the proof.
\endproof

\section*{Acknowledgements}
We would like to thank Xiangqian Yang for some useful discussions.

This work is supported by the National Natural Science Foundation of China Grant No. 12271007.

\bibliographystyle{alpha}
\bibliography{reference}

\end{spacing}
\end{document}